\documentclass[11pt,reqno]{amsproc}
\usepackage[margin=1in]{geometry}
\usepackage{amsmath, amsthm, amssymb}
\usepackage{times, esint,stackrel}
\usepackage{enumitem, graphicx}
\usepackage{color}
\usepackage{enumitem}   
\usepackage{hyperref}
\usepackage{etoolbox}
\usepackage{mathrsfs}
\usepackage{framed}
\usepackage{subcaption,graphicx}

\usepackage{color} 
\usepackage{xcolor} 
\usepackage{hyperref}

\makeatletter
\patchcmd\@thm
{\let\thm@indent\indent}{\let\thm@indent\noindent}%
  {}{}
\makeatother

\usepackage{etoolbox}
\expandafter\patchcmd\csname\string\proof\endcsname
{\normalparindent}{0pt }{}{}

\newenvironment{myquote}{\list{}{\leftmargin=2.2in\rightmargin=0in}\item[]}{\endlist}

\newcommand{\be}{\begin{equation}}
\newcommand{\ee}{\end{equation}}
\newcommand{\bea}{\begin{eqnarray}}
\newcommand{\eea}{\end{eqnarray}}
\newtheorem{thm}{Theorem}

\newtheorem{lemma}{Lemma}

\newtheorem{quest}{Question}

\newtheorem{example}{Example}
\theoremstyle{definition}
\newtheorem{remark}{Remark}

\newcommand{\rmd}{{\rm d}}

\newcommand{\bq}{\begin{equation}}
\newcommand{\eq}{\end{equation}}
\newcommand{\bqa}{\begin{eqnarray*}}
\newcommand{\eqa}{\end{eqnarray*}}

\def\beq{\begin{equation}}
\def\eeq{\end{equation}}

\title[Conjugate  and cut points in ideal fluid motion]{ Conjugate and cut points in ideal fluid motion}
\author{Theodore D. Drivas}
\address{ Department of Mathematics, Stony Brook University,
Stony Brook, NY, 11794}
\email{tdrivas@math.stonybrook.edu}
\author{Gerard Misio{\l}ek} 
\address{Department of Mathematics, University of Notre Dame, IN 46556, USA} 
\email{gmisiole@nd.edu} 
\author{Bin Shi}
\address{RISE-Lab, University of California, Berkeley,  CA 94720-1776}
\email{binshi@berkeley.edu}
\author{Tsuyoshi Yoneda} 
\address{Graduate School of Mathematical Sciences, University of Tokyo,  Komaba 3-8-1 Meguro, 
Tokyo 153-8914, Japan} 
\email{yoneda@ms.u-tokyo.ac.jp} 
\date{\today} 

\begin{document}

\begin{abstract}
Two fluid configurations along a flow are conjugate if there is a one parameter family of geodesics (fluid flows) joining them to infinitesimal order.
 Geometrically, they can be seen as a consequence of the (infinite dimensional) group of volume preserving diffeomorphisms having sufficiently strong positive curvatures which `pull' nearby flows together.
Physically, they indicate a form of (transient) stability in the configuration space of particle positions: a family of flows starting with the same configuration deviate initially and subsequently re-converge (resonate) with each other at some later moment in time.  
Here, we first establish existence of conjugate points in an infinite family of Kolmogorov flows 
-- a class of stationary solutions of  the Euler equations -- on the rectangular flat torus of any aspect ratio. 
The analysis is facilitated by a general criterion for identifying conjugate points 
in the group of volume preserving diffeomorphisms.  
Next, we show non-existence of conjugate points along Arnold stable steady states 
on the annulus, disk and channel.  
Finally, we discuss cut points, their relation to non-injectivity of the exponential map 
(impossibility of determining a flow from a particle configuration at a given instant) 
and show that the closest cut point to the identity is either a conjugate point or 
the midpoint of a time periodic Lagrangian fluid flow.
\end{abstract}

\maketitle

\vspace{-7mm}
\begin{myquote}
{\hfill \emph{To Sasha, with admiration and respect.}}\\
\end{myquote}
\vspace{-3mm}

\section{Introduction} 

Hydrodynamics of ideal fluids is a beautiful example of an infinite-dimensional Riemannian geometry 
where fluid motions in a fixed domain $M$ 
are regarded as geodesics in the group of volume-preserving  diffeomorphisms 
$\mathscr{D}_\mu(M)$
equipped with a right-invariant (weak) Riemannian metric defined by the fluid's kinetic energy. 
This observation was made in a pioneering paper of Arnold \cite{arn} 
who laid out the foundations of a modern differential geometric framework. 
His work was soon followed by \cite{ebma}, \cite{omo}, \cite{shn0} and numerous others 
ranging from applications to classical problems of hydrodynamic stability, 
to extensive calculations of the curvature of various diffeomorphism groups, 
to the search for conjugate points 
and the study of analytic and geometric properties of the associated Riemannian exponential map. 
Since then the field has been under intensive study and many of the subsequent results 
can be found in the recent monograph of Arnold and Khesin \cite{arkh}. 

Arnold was led to this geometric picture through efforts to understand turbulence\footnote{See for example 
his 1982 Crafoord Prize lecture.}. 
The motivation came from Kolmogorov's famous 1958 seminar \cite{am60},
where he proposed to systematically study the problem of long-time, high-Reynolds number fluid
flows driven by forcing of strength commensurate with the viscosity (inverse Reynolds number).  
In this slightly viscous regime, it is natural to try to understand the structure of hydrodynamic attractors 
by means of the ideal (inviscid and incompressible) Euler equations.  
In Arnold's picture, guided by results in finite dimensions due to Smale, Anosov, Sinai among others, 
turbulence could be rationalized if these attractors were chaotic. 
Chaotic attractors are often characterized by extreme sensitivity to initial conditions 
(leading to practical impossibility of pathwise forecasting) and structural stability (allowing robust statistical predictions). 
Averaged characteristics of turbulence observed in experiments, 
such as those underlying the celebrated Kolmogorov 1941 theory and its cousins, 
might then be deduced from a deeper understanding of these attractors.

A canonical example of complex/chaotic behavior in a finite dimensional dynamical system 
is provided by the geodesic flow on a negatively curved surface. 
There, neighboring geodesics diverge exponentially and some may fill out the whole surface densely 
\cite{A69,H71}. 
Inspired by this picture, Arnold's initial hope was to identify attracting manifolds in the configuration space 
of particle positions for the Navier-Stokes dynamics which are covered by exponentially diverging orbits. 
The connection with and the appeal of the geometric approach is now clear: 
since Euler flows are geodesics on the infinite dimensional group of volume-preserving diffeomorphisms, 
curvature calculations can provide a means of identifying regions in the configuration space 
where these attracting manifolds may live. 
Somewhat surprisingly, it turned out that when the fluid domain $M$ is the flat two-torus $\mathbb{T}^2$ 
then sectional curvatures of $\mathscr{D}_\mu(\mathbb{T}^2)$, though negative in ``most directions", 
can take on either sign. 
However, we should point out here that the tangent vectors of the two-planes of positive curvature 
which Arnold discovered are just autonomous incompressible velocities 
and \emph{not} stationary solutions of the Euler equations. 
Thus, while they do shed light on the geometry of $\mathscr{D}_\mu(\mathbb{T}^2)$, 
these curvature calculations may not be directly relevant for the behavior of slightly viscous fluids 
and their attractors. 

The failure of the curvature to be non-positive led naturally to questions concerning 
geodesic deviation involving Jacobi fields and existence of conjugate points along geodesics in the group. 
From a physics point of view, conjugate points may represent a kind of ``relative Lagrangian stability" 
or ``resonance" in the particle configurations along different orbits. 
Namely, a family of fluid flows all starting at the identity diffeomorphism with slightly different initial velocities 
will diverge for a while at first but eventually will be pulled together due to positive curvature 
around the ``conjugate" point in configuration space. At this point, all such flows will have their Lagrangian particles 
-- think of markers of non-diffusive dye -- in roughly the same configuration. This is markedly different 
from the behavior on chaotic hyperbolic attractors, such as geodesic motion on negatively curved manifolds.

The fact that sectional curvatures of the diffeomorphism group are mostly negative makes it easier 
to identify fluid flows without conjugate points. 
Typical examples are the constant-pressure flows on domains of nonpositive curvature, see \cite{mis1}, 
as well as rotations of the unit disk in the plane \cite{pre2}. 
On the other hand, finding conjugate points along geodesics in diffeomorphism groups 
appears to be a harder task. 
While in 3D they seem to abound, 
see e.g., \cite{emp},  \cite{prewash}, \cite{shn}, 
in 2D they have been found so far only along rotations of the unit sphere in $\mathbb{R}^3$, see \cite{mis1}, 
certain zonal flows on the ellipsoid in $\mathbb{R}^3$  \cite{tauyon}, partial results with the Coriolis force on the
sphere \cite{tauyon2}
and, in the case when the fluid domain is flat, 
on a single stationary flow of the two-torus $\mathbb{T}^2$, see \cite{mis2}. 

Of particular interest among the stationary solutions of the Euler equations in 2D 
are the Kolmogorov flows which correspond to eigenfunctions of the Stokes operator on $M$. 
They were proposed by Kolmogorov as the ansatz for the forcing in the aforementioned setup 
to understand turbulence, see \cite{am60}. 
In this case, there always exists a `laminar' solution in the direction of the forcing, 
which is asymptotically stable at low Reynolds numbers 
and becomes unstable at sufficiently high Reynolds numbers \cite{MS61}.\footnote{The only exception 
is when the eigenfunction corresponds to the smallest eigenvalue on the domain. 
In this case, the laminar solution is the unique stationary solution at all Reynolds number which, moreover, 
attracts all orbits at long times.} 
Consequently, this laminar solution, along with its unstable manifold, is embedded in a hydrodynamic attractor 
for all Reynolds numbers, cf. \cite{BV83}. 
This strongly suggests\footnote{In fact, (a minor modification of) the work of Chepyzhov, Vishik and Zelik  \cite{cvz11} 
shows that the sequence of global attractors $\{\mathcal{A}^\nu\}_{\nu>0}$ for 2D Navier-Stokes with forcing strength 
commensurate with the viscosity $\nu$ converges to a so-called trajectory attractor for the unforced Euler equations. 
See also the work of Glatt-Holtz, Sverak and Vicol \cite{gsv15} for an analogous statement 
in a stochastically forced setting. 
This attractor is known to be supported on $H^1$ velocity fields 
whose corresponding vorticity is bounded (the Yudovich space). It is not expected to be supported on 
velocity fields with Sobolev regularity $s>n/2+1$, which is the setting in which the Euler equations 
can currently be understood geometrically. 
Thus, the precise connection with $\mathscr{D}_\mu^s(M)$ is not yet clear 
and building a geometric theory for Yudovich solutions may be of crucial importance 
for advancing our understanding of high-Reynolds number 2D turbulence.} 
that obtaining a clear picture of the geometry 
near the corresponding geodesic in $\mathscr{D}_\mu(M)$ 
will provide a window into the structure of a subset of the attractor at `infinite Reynolds number'. 
In fact, 
possibly for that reason, 
Arnold proposed (see \cite{arn-P}) to search for conjugate points within the Kolmogorov family. 
The existence of conjugate points would show that this region of the attractor is not uniformly hyperbolic 
and indicate a form of Lagrangian stability, making the long-time, small-viscosity dynamics 
more subtle than initially expected. 

In this paper we focus on two-dimensional fluids. 
First, we revisit the criterion introduced in \cite{mis2} and applied there to obtain the first example of 
a conjugate point in $\mathscr{D}_\mu(\mathbb{T}^2)$. 
This criterion generalizes readily to arbitrary Riemannian fluid domains, possibly with boundary. 
Using it we produce an infinite family of Kolmogorov flows admitting conjugate points 
when the domain is a rectangular torus $\mathbb{T}^2_\alpha$ with any aspect ratio $\alpha>0$. 
As the aspect ratio increases (elongating the domain in either direction) the family extends 
to include (nearly) the entire class of Kolmogorov flows. 
It is tempting to conjecture that, with the exception of unidirectional shear flows 
with the sinusoidal (or `cosinusoidal') velocity profiles, every Kolmogorov flow on $\mathbb{T}^2_\alpha$ 
gives rise to a stationary geodesic with conjugate points in the group $\mathscr{D}_\mu(\mathbb{T}^2_\alpha)$.

Next, we pursue the relation between Arnold's nonlinear stability criterion 
and conjugate points in the group $\mathscr{D}_\mu(M)$. 
According to Arnold's criterion, a stationary flow of an ideal fluid is Lyapunov stable 
if the quadratic form given by the second variation of the kinetic energy restricted to coadjoint orbits 
in the algebra of smooth divergence-free vector fields is positive or negative definite, 
see e.g., \cite{arn}, \cite{arkh}, \cite{bur}. 
Using strong rigidity properties of \cite{cdg} together with known conditions for non-conjugacy, 
we deduce that for a  number of simple domains including the disk, annulus and straight channel, 
no Arnold stable steady states possess conjugate points. 
In light of these examples, we speculate that this should be true more generally.

Finally, we investigate the relation between cut points along fluid trajectories, defined and studied by Shnirelman, 
and the conjugate points. Roughly, a cut point is the first point beyond which a given geodesic 
no longer minimizes the energy functional among all paths connecting it to the identity. 
One distinction between a cut point and a conjugate point is that the former does not require any information 
on growth of nearby deviations (studied via the Jacobi fields) and therefore need not imply any 
``relative Lagrangian stability" for an infinite family of fluid flows. 
A simple example of a cut point which is not a conjugate point is a solid body rotation of the unit disk. 
To further illustrate their relation we also prove a geometric property of the cut locus 
which is similar to that in finite dimensions.

\section{The kinetic energy metric, its exponential map and  geodesics} 
\label{sec:KE} 

First, we briefly review the basic differential-geometric constructions of hydrodynamics 
and recall some known results we will need in later sections. 
Let $(M, g)$ be a compact two-dimensional Riemannian manifold possibly with boundary, 
e.g., the flat two-torus $\mathbb{T}^2$, the sphere in $\mathbb{R}^3$ with the round metric 
or a two-dimensional bounded domain with smooth boundary. 
Let $\mathscr{D}_\mu^s$ denote the completion in the Sobolev $H^s$ topology of 
the group of all $C^\infty$ diffeomorphisms of $M$ preserving its Riemannian volume form $\mu$. 
If $s > 2$ then it is well known that $\mathscr{D}_\mu^s$ is a submanifold of the smooth Hilbert manifold 
$\mathscr{D}^s$ of all $H^s$ diffeomorphisms of $M$. 
Both spaces can be equipped with a (weak) Riemannian metric which at a point $\eta$ 
is given by 
\begin{equation} \label{eq:L2metric} 
\langle V, W \rangle_{L^2} 
= 
\int_M \langle V(x), W(x) \rangle_{\eta(x)} d\mu 
\end{equation} 
for any $V, W \in T_\eta\mathscr{D}^s$. 
This metric is invariant under the action of $\mathscr{D}_\mu^s$ by right multiplications 
and corresponds to the fluid's total kinetic energy. 
According to the ``least action principle", its geodesics are the critical points 
that minimize the associated $L^2$ energy functional 
\begin{align} \label{eq:energy} 
\mathcal{E}(\gamma) 
= 
\frac{1}{2} \int_0^T \| \dot{\gamma}(t)\|_{L^2}^2 dt 
\end{align} 
of a path $\gamma(t)$ in $\mathscr{D}_\mu^s$ between $t=0$ and $t=T$ 
or, alternatively, the $L^2$ length functional 
\begin{align} \label{eq:length} 
\mathcal{L}(\gamma) 
= 
\int_0^T \| \dot{\gamma}(t)\|_{L^2} dt 
\end{align}
if $\gamma$ is given the $L^2$ unit speed parametrization. 
The geodesic equations when right-translated to the Lie algebra $T_e\mathscr{D}_\mu^s$ 
of divergence-free vector fields 
(i.e., the tangent space at the identity diffeomorphism $e$) 
take on the familiar form of the Euler equations of an incompressible and inviscid fluid in $M$, namely 
\begin{align} \label{eq:Euler} 
&u_t + u\cdot \nabla u = - \mathrm{grad}\, p 
\\ 
&\mathrm{div}\, u = 0 
\end{align} 
where $u = \dot{\gamma}\circ\gamma^{-1}$ is the velocity field of the fluid 
and $u{\cdot}\nabla$ denotes the covariant derivative of the metric $g$ in the direction of $u$. 
When boundaries are present then we require, in addition, that $u{\parallel}\partial M$. 
The pressure gradient on the right hand side is determined uniquely by the incompressibility constraint as 
$\mathrm{grad}\, p = Q_e(u{\cdot}\nabla u)$ 
where the formulas 
\begin{equation} \label{eq:Q} 
Q_e = \mathrm{grad}\, \Delta^{-1} \mathrm{div} = \mathrm{Id} - P_e 
\end{equation} 
define the usual $L^2$ orthogonal projections $P$ and $Q$ at the identity $e$ 
onto the subspace of gradient and divergence-free vector fields 
in the Weyl-Helmholtz-Hodge decomposition on $M$. 

We should point out that the Lagrangian formulation of fluid equations 
(as geodesic equations in a Banach manifold) 
has a technical advantage over the Eulerian (as the system \eqref{eq:Euler} of nonlinear PDE). 
Since the geodesic equations in $\mathscr{D}_\mu^s$ are an ODE their solutions are obtained 
using Picard's method of successive approximations 
and as such they depend smoothly on initial data. 
Consequently, 
the right-invariant $L^2$ metric possesses a smooth exponential map 
along with a Levi-Civita connection $\nabla$
and a curvature tensor $\mathcal{R}$ which is a bounded multilinear operator 
on each tangent space to $\mathscr{D}_\mu^s$. 
The $L^2$ exponential map is defined in the standard way as the solution map of 
the geodesic equations: 
it maps lines through the origin in a tangent space at a given fluid configuration (fluid velocities) 
onto geodesics in the diffeomorphism group (fluid flows). 
More precisely, at the identity diffeomorphism, we set 
\begin{equation} \label{eq:exp} 
\mathrm{exp}_e : T_e\mathscr{D}_\mu^s \to \mathscr{D}_\mu^s, 
\qquad 
\mathrm{exp}_e{tu_0} = \gamma(t), 
\quad 
t \in \mathbb{R} 
\end{equation} 
where $\gamma(t)$ is the unique $L^2$ geodesic starting from $e$ with velocity $u_0$. 
As in finite dimensions $\mathrm{exp}_e$ is a local diffeomorphism near $e$ by the inverse function theorem. 
Since $M$ is two-dimensional it is defined on the whole tangent space, 
essentially as a consequence of the classical result of Wolibner \cite{wol}. 
In fact, Shnirelman \cite{shn12} proved that $\exp_e$ is a real analytic map 
provided that the fluid domain is an analytic manifold.  
As a corollary he recovered the statement that Lagrangian trajectories are analytic functions of time 
at fixed labels, which apparently was overlooked but implicit in the work of Lichtenstein\footnote{A. Shnirelman, private communication.} \cite{Lich25},  proved first in the literature by Serfati \cite{ser} and subsequently extended by various authors 
(see e.g. \cite{ckv,frisch}). He also showed
that streamlines of stationary Euler solutions are analytic curves, 
extending work of Nadirashvili \cite{Nad}. 

Roughly speaking, \emph{conjugate points} are the singular values of the exponential map 
and \emph{conjugate vectors} are the corresponding elements of the tangent space. 
In general, in infinite dimensions one has to contend with two different kinds of these points 
depending on whether the derivative of the exponential map fails to be one-to-one (a monoconjugate point) 
or fails to be onto (an epiconjugate point). 
However, while both species occur along fluid motions in 3D this is not so in 2D hydrodynamics 
because in this case the $L^2$ exponential map is known to be a nonlinear Fredholm map of index zero, 
see \cite{emp}. 
In particular, mono- and epiconjugate points coincide and therefore display many of the same properties 
as their counterparts in the finite dimensional Riemannian geometry 
- 
in the sequel we will refer to them simply as conjugate points. 
The set of all conjugate points of the identity in $\mathscr{D}_\mu^s$ is called the \emph{conjugate locus} of $e$ 
and the set of the corresponding vectors in $T_e\mathscr{D}_\mu^s$ will be denoted by $\mathcal{C}_e$. 

The Fredholm property of the exponential map on $\mathscr{D}_\mu^s$ has several important consequences. 
Below we list some of them along with other results that we shall need in what follows. 
\begin{enumerate} 
\item[(i)](\emph{The Index Theorem in 2D hydrodynamics}) 
If $\gamma(s,t)$ ($-\varepsilon<s<\varepsilon$) is a variation with fixed ends 
of an $L^2$ geodesic $\gamma(t)$ from the identity in $\mathscr{D}_\mu^s$ 
and $V(t) = \partial \gamma/\partial s (0,t)$ 
is the associated variation field 
then the first and the second variations of the energy functional read 
\begin{align} \label{eq:var-1} 
\mathcal{E}'(\gamma)(V) 
= 
\langle V, \dot{\gamma} \rangle_{L^2} \big|_{_{t=0}}^{^{t=T}} 
- 
\int_0^T \langle V, \nabla_{\dot{\gamma}}\dot{\gamma} \rangle_{L^2} dt 
\end{align} 
and 
\begin{align} \label{eq:var-2} 
\mathcal{E}''(\gamma)(V,V) 
= 
\int_0^T \big( \| \nabla_{\dot{\gamma}} V \|_{L^2}^2 
- 
\langle \mathcal{R}_\gamma(V, \dot{\gamma})\dot{\gamma}, V \rangle_{L^2} \big) dt 
\end{align} 
where $\mathcal{R}_\gamma$ is the $L^2$ curvature tensor along $\gamma$. 
There is also a corresponding formula for $\mathcal{L}$. 
The integral in \eqref{eq:var-2} induces on the space of vector fields on $\gamma$ 
that vanish at the endpoints, 
a symmetric bilinear form called the \emph{index form}. 
The maximum dimension of the subspace 
on which the index form is negative definite is finite for any finite geodesic segment 
and equals the number of conjugate points to $e$ along $\gamma$ 
counted with multiplicity\footnote{The multiplicity of a conjugate point is the dimension of 
the kernel of the differential of $\mathrm{exp}_e$ at that point.}, 
cf. \cite{mispre}. 
Among immediate consequences is that conjugate points along fluid flows 
in 2D hydrodynamics are isolated and that flows do not minimize the $L^2$ length 
past their first conjugate point. 
\item[(ii)](\emph{Covering properties of $\mathrm{exp}_e$}) 
The $L^2$ exponential map is a covering space map on the open connected component 
$\mathcal{U} \subset \mathscr{D}_\mu^s$ of the identity $e$ whose $L^2$ diameter is infinite, 
cf. \cite{mispre}. 
While it is not yet known whether this connected component is in fact the whole group, 
if resolved in the affirmative, it could be considered as a hydrodynamical Hopf-Rinow theorem. 
\end{enumerate} 

A related notion in classical Riemannian geometry is that of a \emph{cut point}
defined as the first point beyond which a geodesic starting from a given point ceases to minimize 
the length functional (or the energy functional) among all paths connecting the two points. 
As such, it always appears before or at the first conjugate point. 
The set of all cut points of a given point is called its \emph{cut locus}. 
In Chapter 4 of \cite{arkh}, Shnirelman introduced this notion for the $L^2$ geodesics. 
He also called such a point a (first) \textit{local cut point} if shorter paths can be chosen arbitrarily close to 
the given geodesic in the manifold topology. 
He proved that any sufficiently long geodesic in $\mathscr{D}_\mu^s(M)$ 
will contain a local cut point if $\dim{M}=3$. 
The argument uses generalized flows of Brenier, see \cite{shn}. 

Shnirelman also points out that a precise relation between cut points and conjugate points 
in geometric hydrodynamics remains unclear. 
In fact, even in finite dimensions the structure of the cut locus is highly complicated and nontrivial to describe. 
In infinite dimensions the difficulties are further compounded. 
For example, the known classical proofs of the characterization of a cut point as 
either the first conjugate point or the meeting point of distinct minimizing geodesics 
rely on local compactness and do not carry over directly to the setting of diffeomorphism groups.
Nevertheless, something interesting can be said about these notions in the 2D case. 
\begin{enumerate} 
\item[(iii)](\emph{Non-injectivity of $\mathrm{exp}_e$ near conjugate points}) 
The $L^2$ exponential map is never one-to-one in any neighbourhood of its conjugate point, 
cf. \cite{mis3}. 
This statement can be viewed as a hydrodynamical analogue of the classical theorem of Morse and Littauer. 
It has the following interpretation for 2D fluid motions: 
in every $H^s$ neighborhood of a conjugate point there will be fluid configurations 
which can be reached from the initial configuration $e$ by at least two distinct fluid flows 
in the same amount of time. 
\end{enumerate} 
\begin{figure}[h!]
   \centering
\includegraphics[width=0.45\textwidth]{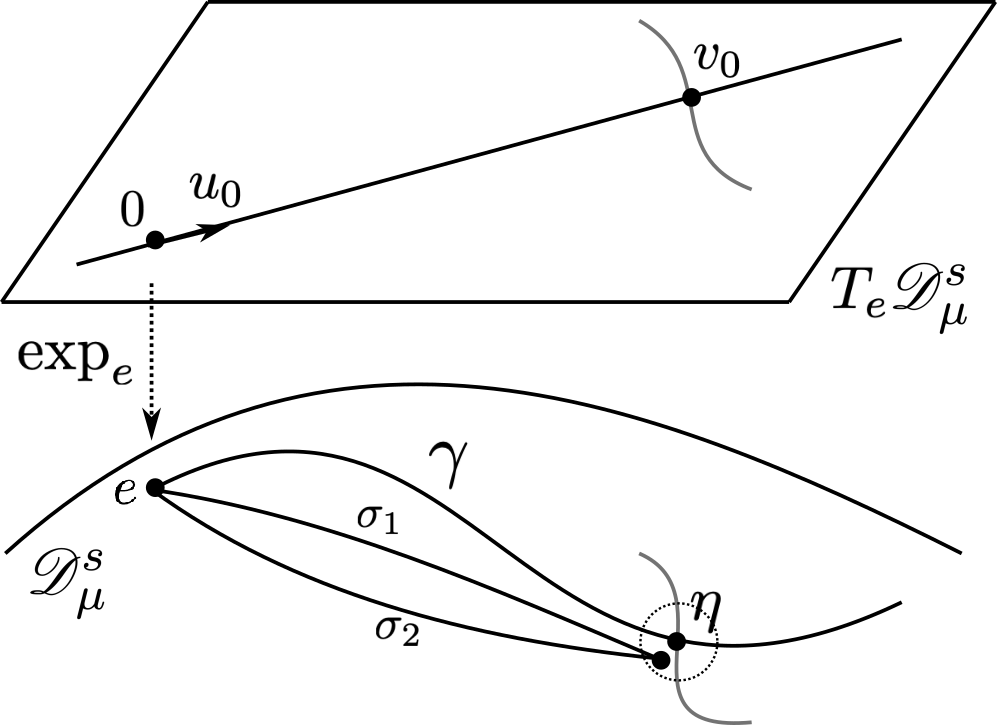}
    \caption{At least two distinct geodesics $\sigma_1$, $\sigma_2$ 
                  intersecting near a conjugate point; cf. Remark \ref{rem:1}.}   \label{fig:thm3}
\end{figure} 
\begin{remark} \label{rem:1} 
In fact, we can say more. 
As a simple application of the covering properties (ii) and non-injectivity of the exp map (iii) 
we can show that 
if $\eta$ is the first conjugate point to $e$ on a geodesic $\gamma(t)$ in $\mathcal{U}$ 
then there are infinitely many geodesics from $e$ which in any $H^s$ neighbourhood of $\eta$ 
meet another geodesic in the same amount of time. 
To prove this one can assume that $\eta = \mathrm{exp}_e{v_0}$ is the image in $\mathcal{U}$ 
of a \emph{regular} conjugate vector in $\mathcal{C}_e$, that is, 
one that has a neighbourhood in $T_e\mathscr{D}_\mu^s$ whose intersection with each ray from the origin 
(if nonempty) contains precisely one conjugate vector of the same multiplicity as $v_0$. 
For example, all conjugate vectors of multiplicity one are of this type. 
Using the Index theorem (i) one can show that
the set of all regular conjugate vectors of $\mathrm{exp}_e$ is a smooth codimension 1 
submanifold of $T_e\mathscr{D}_\mu^s$ which forms an open and dense subset of 
the full conjugate locus $\mathcal{C}_e$, see \cite{lich}. 
It then follows that we can pick a sequence $v_{0,n}$ in $\mathcal{C}_e$ converging to $v_0$ 
in the $H^s$ norm and 
applying (ii) and (iii) for each $n$ we find that the $L^2$ exponential map 
cannot be injective near any of the conjugate points $\eta_n = \mathrm{exp}_e{v_{0n}}$ 
along the corresponding geodesics $\gamma_n(t)$ in $\mathscr{D}_\mu^s$. 
\end{remark}

A different result concerning the cut locus and closed geodesics in $\mathscr{D}_\mu^s$ 
will be proved in Section \ref{sec:cut} below. 

\smallskip 
\noindent One final fact concerning the $L^2$ geodesics in $\mathscr{D}_\mu^s$ 
(or equivalently, fluid motions in $M$) is 
\begin{enumerate} 
\item[(iv)](\emph{Absence of self-intersecting geodesics}) 
The result of Ebin \cite{eb} shows that geodesics in $\mathscr{D}_\mu^s(M)$ 
cannot have self-intersections at least when the underlying manifold $M$ is the flat 2-torus $\mathbb{T}^2$ 
or a simply connected planar domain with smooth boundary.
It also holds for domains that admit at least one Killing field. 
It is not known if this property has any global implications for the geometry of the right-invariant $L^2$ metric 
on $\mathscr{D}_\mu^s$. 
\end{enumerate} 
%

\section{A sufficient condition for existence of conjugate points} 
\label{sec:conj_points} 

There are two principal methods for computing conjugate points along geodesics. 
The first looks for nontrivial solutions of the associated Jacobi equation that vanish at two distinct points. 
The second, more geometric, relies on constructing a one-parameter family of geodesics 
which come together at two end points to infinitesimal order. 
In fact, either method can be used to define a (mono-) conjugate point 
but neither is easily applicable in the case of diffeomorphism groups. 
While they can be used to good effect in certain special situations, 
the former typically gets bogged down quickly in complicated calculations 
and the latter requires some prior insight into the geometry of the neighbourhood of the geodesic. 

We may settle for less and look for a way of detecting conjugate points 
without trying to precisely locate them along a given geodesic. 
Classical Riemannian geometry developed methods to control the metric and its geodesics 
by means of suitable curvature bounds. 
However, standard techniques such as those based on Rauch's Comparison theorems 
are only of limited use in hydrodynamics. 
For example, it is known that for any underlying manifold 
the $L^2$ curvature of the group of volume-preserving diffeomorphisms $\mathscr{D}_\mu^s$ 
can assume positive as well as negative signs. 
In fact, given any $u \in T_e\mathscr{D}_\mu^s$ the minimum value of the sectional curvature 
on the two-planes containing $u$ is always strictly negative, 
unless $u$ is a Killing field, in which case it is zero, cf. \cite{klmp}; Thm 5.1. 
In light of this fact we turn elsewhere for help in establishing existence of conjugate points. 
One useful sufficient condition can be extracted from the method devised in \cite{mis2} 
to produce the first example of a conjugate point in $\mathscr{D}_\mu^s(\mathbb{T}^2)$. 
This condition was recently revisited and applied in \cite{tauyon} to zonal flows 
on an ellipsoid in $\mathbb{R}^3$. More precisely, we have
\begin{thm} \label{conjpointthm} 
Let $M$ be a compact two-dimensional Riemannian manifold. 
Let $u_0$ be a stationary solution of the Euler equations \eqref{eq:Euler}. 
Suppose that there is $v \in T_e\mathscr{D}_\mu^s$ with $ \| v \|_{L^2}=1$ 
such that 
\begin{align} \label{eqn:mc} 
m_c^{u_0,v} 
&:= 
\big\langle [u_0,v] {\cdot} \nabla u_0 + u_0 {\cdot} \nabla [u_0,v], v \big\rangle_{L^2} > 0.
\end{align} 
Then the geodesic in $\mathscr{D}_\mu^s$ from the identity $e$ in the direction $u_0$ 
has a conjugate point. 
Moreover, this conjugate point must occur at a time $t \leq t_c:=  \pi \sqrt{2/m_c^{u_0,v} }$.
\end{thm} 
\begin{remark} 
In fact, as will be seen in the proof below, we have 
\be
m_c^{u_0,v} 
= 
\langle \mathcal{R}_e(v, u_0) u_0, v\rangle_{L^2} 
- 
\| P_e ( u_0{\cdot} \nabla v)\|_{L^2}^2
\ee 
where $ \mathcal{R}_e$ is the $L^2$ curvature tensor of $\mathscr{D}^s_\mu$ at the identity $e$ 
and $P_e$ is the projector onto divergence-free vector fields.
Observe that positivity of $m_c^{u_0,v}$ translates into a strict lower bound 
on the sectional curvature away from zero. 
It is worth recalling the example of a paraboloid of revolution in $\mathbb{R}^3$ 
which has strictly positive sectional curvature but with no uniform lower bound and consequently  geodesics emanating  from the vertex have no conjugate points. 
\end{remark} 
\begin{proof} 
Let $\gamma(t)$ be the one-parameter subgroup in $\mathscr{D}_\mu^s$ generated by $u_0$ 
and hence also a critical point of the energy functional $\mathcal{E}$ 
with respect to any variation of $\gamma$ with fixed end points. 
The basic idea is to produce a vector field along $\gamma$ and using the group structure and 
some submanifold geometry show that the second variation of $\mathcal{E}$ 
evaluated on this vector field is strictly negative. 

To this end, recall that on the full diffeomorphism group $\mathscr{D}^s$ 
the $L^2$ inner product \eqref{eq:L2metric} also induces a smooth Levi-Civita connection $\bar{\nabla}$ 
whose sectional curvature on any 2-plane in $T_e\mathscr{D}^s$ is completely determined 
by that of the underlying Riemannian manifold 
and given explicitly as the integral over $M$ of its pointwise sectional curvature $R^{_M}$. 
Furthermore, since $\mathscr{D}_\mu^s$ is a smooth submanifold of $\mathscr{D}^s$, 
the sectional curvatures of the two groups are related by the $L^2$ analogues of the classical Gauss-Codazzi equations. 
The derivation is routine as in the finite dimensional case. 
Specifically, 
for any right invariant vector fields $V$ and $W$ on $\mathscr{D}_\mu^s$ 
the connections $\bar{\nabla}$ and $\nabla = P\bar{\nabla}$ 
are related by the Weyl-Helmholtz projection $P$ 
and their difference is the second fundamental form $\mathrm{II}$ of $\mathscr{D}_\mu^s$, namely 
\begin{equation} \label{eq:GaussF} 
(\bar{\nabla}_VW)_\eta = (\nabla_VW)_\eta + \mathrm{II}_\eta(V,W), 
\qquad 
\eta \in \mathscr{D}_\mu^s 
\end{equation} 
where 
$(\bar{\nabla}_VW)_\eta = (v{\cdot}\nabla w) \circ \eta$ 
with $V=v\circ\eta$, $W=w\circ\eta$ and $v,w \in T_e\mathscr{D}_\mu^s$. 
In particular, we find that 
$\mathrm{II}_e(v, w) = Q_e(v{\cdot}\nabla w)$ 
is just the polarization of the pressure gradient term 
(see \eqref{eq:Q} above). 
Inserting formula \eqref{eq:GaussF} (with $w=u_0$) into the expression for the curvature tensor 
$\bar{\mathcal{R}}_e$ of $\mathscr{D}^s$ and taking the $L^2$ product with $v$ 
we obtain 
\begin{align} \nonumber 
\big\langle \mathcal{R}_e(v,u_0) u_0, v \big\rangle_{L^2} 
&= 
\big\langle \bar{\mathcal{R}}_e(v,u_0)u_0,v \big\rangle_{L^2} 
+ 
\big\langle \mathrm{II}_e(u_0,u_0), \mathrm{II}_e(v,v) \big\rangle_{L^2} 
- 
\big\langle \mathrm{II}_e(u_0, v), \mathrm{II}_e(v,u_0) \big\rangle_{L^2} 
\\ \label{eq:Gauss} 
&= 
\big\langle R^{_M}(v, u_0) u_0, v \big\rangle_{L^2} 
+ 
\big\langle Q_e(u_0{\cdot}\nabla u_0), Q_e(v{\cdot}\nabla v) \big\rangle_{L^2} 
- 
\big\| Q_e(u_0{\cdot}\nabla v) \big\|_{L^2}^2 
\end{align} 
since the second fundamental form $\mathrm{II}$ is symmetric in its entries 
and takes values in the $L^2$ orthogonal complement of $T_e\mathscr{D}_\mu^s$ in $T_e\mathscr{D}^s$. 

The remainder of the argument rests on establishing the following two statements. 
\begin{enumerate} 
\item { \emph{(index estimate for non-conjugacy)}}
Given any $T>0$, if $\gamma(t)$ has no points conjugate to $e$ in the segment $[0, T]$ 
then the index form $\mathcal{E}''(\gamma)$ is nonnegative on any vector field on $\gamma$ 
that vanishes at the two end points. 
\item { \emph{(negativity of index form)}}
There exist $t_c >0$ and a vector field $V(t)$ along $\gamma(t)$ such that 
$V(t)$ vanishes at $t=0$ and $t=t_c$ and $\mathcal{E}''(\gamma)(V,V) < 0$. 
\end{enumerate} 

Regarding (1) recall that $d\, \mathrm{exp}_e(tu_0)$ is a bounded Fredholm operator of index zero 
and so the assumption on conjugate points implies that it is in fact an isomorphism of the Banach spaces 
$T_{tu_0}T_e\mathscr{D}_\mu^s \simeq T_e\mathscr{D}_\mu^s$ and $T_{\gamma(t)}\mathscr{D}_\mu^s$ 
for each $0 \leq t \leq T$. 
Therefore, given any vector field $Z(t)$ vanishing at the end points of the segment $\gamma([0,T])$ 
one can construct a variation family $\gamma_s(t) = \mathrm{exp}_{\gamma(t)}{sZ(t)}$ of $\gamma(t)$ 
for sufficiently small values of the parameter $s$ whose infinitesimal variation field 
is precisely the given vector field $Z(t)$. 
A direct calculation shows that the energy functional satisfies 
$\mathcal{E}(\gamma_s) \geq \mathcal{E}(\gamma)$ 
for all sufficiently small $s$. 
The calculation relies on an $L^2$ version of the Gauss Lemma which follows in turn 
from the first variation formula \eqref{eq:var-1} as in the classical Riemannian case. 
For further details we refer to \cite{mis2}. 

%

Of more interest to us is the second statement (2). 
Given $v$ in $T_e\mathscr{D}_\mu^s$ satisfying the condition in \eqref{eqn:mc} 
with 
$m_c = m_c^{u_0,v}>0$, 
we set 
$t_c = \pi \sqrt{2/m_c}$ 
and using right translations in $\mathscr{D}_\mu^s$ (given by compositions with diffeomorphisms) 
define a vector field 
$V(t) = f(t) v \circ \gamma(t)$ 
where $f(t)$ is a smooth nonzero function with $f(0)=f(t_c)=0$. 
Differentiating $V(t)$ along the geodesic, 
integrating by parts and using right invariance of the $L^2$ metric 
we find 
\begin{equation} \label{eq:induced} 
\| \nabla_{\dot{\gamma}} V \|_{L^2}^2 
= 
f'^2 \|v\|_{L^2}^2 + f^2 \| P_e( u_0{\cdot}\nabla v) \|_{L^2}^2.
\end{equation} 
Substituting \eqref{eq:induced} into \eqref{eq:var-2} 
and employing the  Gauss-Codazzi  equation \eqref{eq:Gauss} 
together with the fact that $P_e$ and $Q_e$ are $L^2$ orthogonal projections 
we obtain 
\begin{align*} 
\mathcal{E}''(\gamma)(V,V) 
&= 
\int_0^{t_c} \Big( 
f'^2 \|v\|_{L^2}^2 
+ 
f^2 \|P_e (u_0{\cdot}\nabla v)\|_{L^2}^2 
- 
f^2  \langle \mathcal{R}_e(v, u_0) u_0, v \rangle_{L^2} \Big) dt 
\\ 
&= 
\int_0^{t_c} \Big( 
f'^2 \|v\|_{L^2}^2 
- 
f^2 \big\langle R^{_M}(v, u_0) u_0, v \big\rangle_{L^2} 
- 
f^2 \big\langle u_0{\cdot}\nabla u_0, v{\cdot}\nabla v \big\rangle_{L^2} 
+ 
f^2 \|u_0{\cdot}\nabla v\|_{L^2}^2 
\Big) dt
\end{align*} 
since $u_0$ is a stationary Euler solution, so that $P_e (u_0 {\cdot} \nabla u_0) = 0$.  
Next, expressing the curvature term with the help of the covariant derivative of the Riemannian metric on $M$ we have
\be
\big\langle R^{_M}(v, u_0) u_0, v \big\rangle_{L^2} =\big\langle \nabla_{v}^{_M}  \nabla_{u_0}^{_M} u_0 -  \nabla_{u_0}^{_M}  \nabla_{v}^{_M}  u_0-  \nabla_{[u_0, v]}^{_M} u_0, v \big\rangle_{L^2}
\ee
where, for simplicity, we have used the alternative notation $\nabla_{u}^{_M} v = u {\cdot}\nabla v$.
Finally, recalling $\|v\|_{L^2}=1$ and manipulating the terms on the right hand side we find
\begin{align*} 
\mathcal{E}''(\gamma)(V,V) = 
\int_0^{t_c} 
\Big( 
f'^2 \|v\|_{L^2}^2 
- 
f^2 \big\langle [u_0,v] {\cdot} \nabla u_0 + u_0 {\cdot} \nabla [u_0,v], v \big\rangle_{L^2} 
\Big) dt 
&= 
\int_0^{t_c} \big( f'^2 - m_c f^2 \big) dt.
\end{align*} 
The proof is completed by taking 
$f(t) = \sin{t\sqrt{m_c/2}}$. 
For more details we refer once again to \cite{mis2}; cf. Lemmas 2 and 3. 
\end{proof} 
%

\section{Existence of conjugate points along Kolmogorov flows} 

In this section, we establish one of our main results which concerns the plethora of 
conjugate points in $\mathscr{D}_\mu^s(\mathbb{T}^2_\alpha)$ 
where $\mathbb{T}^2_\alpha =[0, 2\pi/\alpha)\times [0, 2\pi)$ 
for $\alpha\in  \mathbb{R}^+$ is the flat, rectangular torus.  
Our main theorem will follow as an application of the machinery of \S \ref{sec:conj_points}.  
As discussed in the Introduction, we focus here on finding conjugate points within the family of Kolmogorov flows 
- a problem for investigation proposed by Arnold (Problem 1970-6; \cite{arn-P}). 
Kolmogorov flows are the eigenfunctions of the Stokes operator $A= - P_e \Delta$ 
on $\mathbb{T}^2_\alpha$: 
for some $\lambda>0$ a non-trivial $v_\lambda$ must satisfy
\be \label{seig} 
A v_\lambda = \lambda v_\lambda. 
\ee 
Note that $v_\lambda$ is an eigenfunction of \eqref{seig} with eigenvalue $\lambda$ 
if and only if  $\lambda$ is an eigenvalue of the (negative) Laplacian with a (mean-zero) eigenfunction $q_\lambda$.  Indeed, for an eigenfunction of the Laplacian it is easy to see that the velocity field $v_\lambda= \nabla^\perp q_\lambda$ where $\nabla^\perp = (-\partial_2, \partial_1)$ is  eigenfunction of the Stokes operator.
On the other hand, if $v_\lambda$ is an eigenfunction of $A$, then since it is divergence-free we may write
$v_\lambda = \nabla^\perp q_\lambda$ with $ \int_{\mathbb{T}_\alpha^2} q_\lambda =0$.  Then, since $v_\lambda$ satisfies \eqref{seig} together with the fact that $[\Delta ,P_e]=0$ we obtain $\nabla^\perp (\Delta  q_\lambda + \lambda  q_\lambda)=0$.  Taking the curl, we find $\Delta(\Delta  q_\lambda + \lambda  q_\lambda)=0$ which implies $ -\Delta  q_\lambda = \lambda  q_\lambda$ since $q_\lambda$ is mean zero.
Consequently, the spectrum is 
\be 
\sigma(A) 
= 
\Big\{ (\alpha n)^2 + m^2 :\ n, m\in \mathbb{N}, \ \  (n,m)\neq (0,0)  \Big\}. 
\ee 
There are results about non-existence of conjugate points in this family.
For example, the unidirectional Kolmogorov flows, i.e.  $k_1,k_2\in \mathbb{Z}$
 \be
 \cos(k_1 \alpha x+ k_2 y) \quad \text{or} \quad  \sin(k_1 \alpha x+ k_2 y) 
 \ee
 do not possess conjugate points. 
 These flows are characterized within the Kolmogorov class by having straight streamlines 
 (in contrast with a typical member of the family which has cellular structure, see Figure \ref{fig:stream}a). 
 In this case, Arnold \cite{arn} proved that the sectional curvature was non-positive, 
 see also Lukatski \cite{luk}. 
 The result when either $k_1$ or $k_2$ is trivial follows from prior work on plane parallel flows \cite{mis1}.

In the following theorem, we compute the quantity $m_c$ (cf. equation \eqref{eqn:mc}) 
in a particular direction for a large class of Kolmogorov flows corresponding to eigenfunctions 
$q_\lambda =- \cos(n \alpha  x) \cos(m y)$. 
We find it to be positive for infinitely many $n,m$ implying that sectional curvatures 
in this direction are positive and strong enough to ensure existence of conjugate points 
by Theorem \ref{conjpointthm}. 
We have
\begin{thm}[Conjugate points in $\mathscr{D}_\mu^s(\mathbb{T}^2_\alpha)$] \label{kflow1}
Each Kolmogorov flow defined by the stream function 
\be\label{steadystate}
\psi (x,y) =- \cos(n \alpha  x) \cos(m y)
\ee
for any $n,m\in \mathbb{Z}$ such that
\be\label{const} 
\text{either} \quad 
3 + 11 m^2 + 6 m^4 + (3 - 2 m^2) n^2 \alpha^2 \leq 0  \quad \text{or} \quad 3 + 11 n^2 + 6 n^4 + (3 - 2 n^2) m^2 \alpha^{-2}  \leq 0
\ee
possesses a conjugate point. 
\end{thm}
\begin{remark}
As $\alpha\to \infty$ or $\alpha\to 0$, the condition becomes $n\neq 0$ and $|m|\geq  2$ 
and $m\neq 0$ and $|n|\geq  2$ respectively. 
Therefore the range where conjugate points can be found extends to (nearly) 
the entire admissible range of $(n,m)$ with the notable exception of points 
$(n,m)=(n,\pm 1)$ and $(n,m)=(\pm 1, m)$ respectively. 
The ranges in \eqref{const} is visualized in Figure \ref{fig:range} in two cases. 
\begin{figure}[h!]
   \centering
   \begin{subfigure}[b]{0.3\textwidth}
       \includegraphics[width=\textwidth]{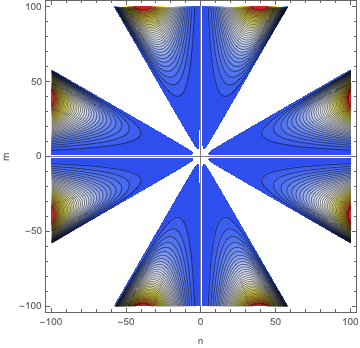}
   \end{subfigure}\hspace{5mm}
   \begin{subfigure}[b]{0.3\textwidth}
       \includegraphics[width=\textwidth]{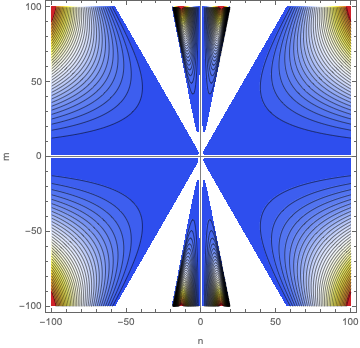}
    \end{subfigure}
  \caption{Contour plots of $m_c^{u_0,v}(\alpha)$ restricted to positive values (smaller values in \textcolor{blue}{blue} and larger in \textcolor{red}{red}). On left:  filled in regions represent those for 
 which $(n,m)\in \mathbb{Z}^2$  satisfy the constraints \eqref{const}. On right: filled regions  contain  $(n,m)\in \mathbb{Z}^2$ satisfing \eqref{const} with $\alpha=3$.}
     \label{fig:range}
\end{figure} 
\end{remark}

\begin{proof} 
Theorem \ref{kflow1} follows from the general method described in \S \ref{sec:conj_points}. 
In particular, with the steady state $u_0=\nabla^\perp \psi$ where $\psi$ is given by \eqref{steadystate} 
with $\alpha\in \mathbb{R}^+$, we define the test stream function 
\be \label{teststream}
\phi(x,y) = -2 \cos(n \alpha x +  y ) \cos(m y), 
\ee 
and $v=\nabla^\perp \phi$. 
See Fig. \ref{fig:stream} for a visualization of the streamfunction plot of the steady state and the test field. 
\begin{remark}
Some heuristic discussion of the choice of the test field \eqref{teststream} might be helpful. 
Roughly, to identify a point on a geodesic as a conjugate point one looks for a shorter path connecting it to the identity. 
Now, the test field chosen above is nothing but a shearing of the original steady state \eqref{steadystate}. 
Shearing of a coherent structure is a mechanism that reduces its energy, see \cite{VY}.  
Thus, this test field slightly perturbs the steady state towards a geodesic with lesser ``action" (energy) 
while at the same time maintaining enough structure of the steady flow without destabilizing it. 
There is some action required to perturb to this state but over time this nearby geodesic of lesser action 
will ``catch up" to the original until a resonance, i.e., our conjugate point, occurs.
\end{remark} 
With this choice, a lengthy but straightforward calculation yields 
\be\label{recm} 
m_c^{u_0,v}(\alpha) 
= 
- \left(\frac{n^2\alpha^2}{8\pi^2}\right) 
\frac{ 3 + 11 m^2 + 6 m^4 + (3 - 2 m^2) n^2 \alpha^2}{ (m^2 + n^2 \alpha^2) (1 + m^2 + n^2 \alpha^2)  }. 
\ee 
The first condition \eqref{const} follows directly from the requirement that $m_c^{u_0,v}(\alpha) >0$. 
The second condition in \eqref{const} comes from using the test stream function 
\be\label{teststream2} 
\phi(x,y) = -2 \cos(n \alpha x ) \cos(m y+ \alpha x), 
\ee 
and $v=\nabla^\perp \phi$, which yields 
\be\label{recm} 
m_c^{u_0,v}(\alpha) 
= 
- \left(\frac{n^2\alpha^4}{8\pi^2}\right)
\frac{  3 + 11 n^2 + 6 n^4 + (3 - 2 n^2) m^2 \alpha^{-2} }{ (m^2 + n^2 \alpha^2) (1 + m^2 + n^2 \alpha^2)  }. 
\ee 
Applying Theorem \ref{conjpointthm} we deduce the existence of conjugate points. 
\begin{figure}
   \centering
   \begin{subfigure}[b]{0.25\textwidth}
       \includegraphics[width=\textwidth]{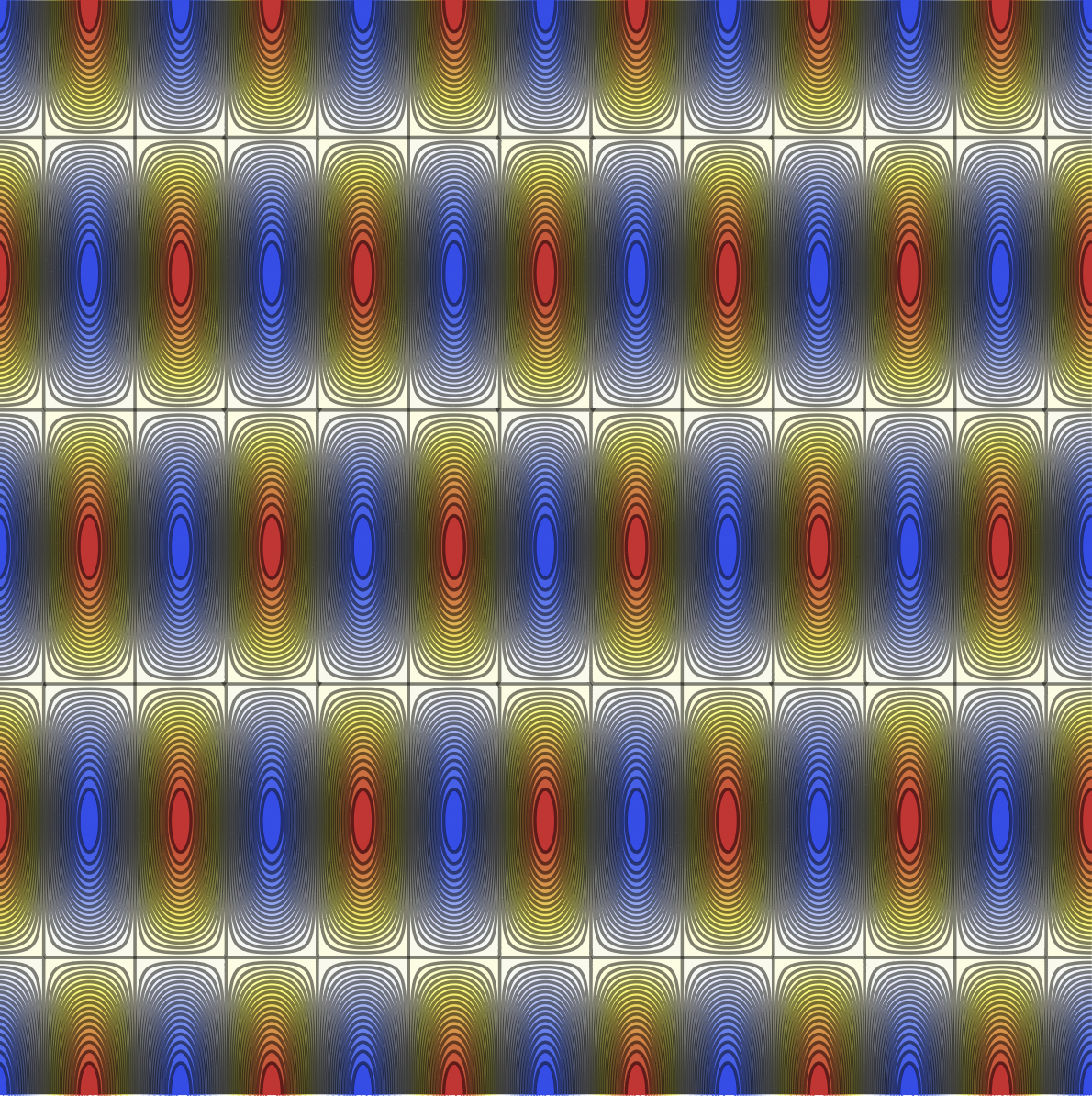}
   \end{subfigure}\hspace{12mm}
   \begin{subfigure}[b]{0.25\textwidth}
       \includegraphics[width=\textwidth]{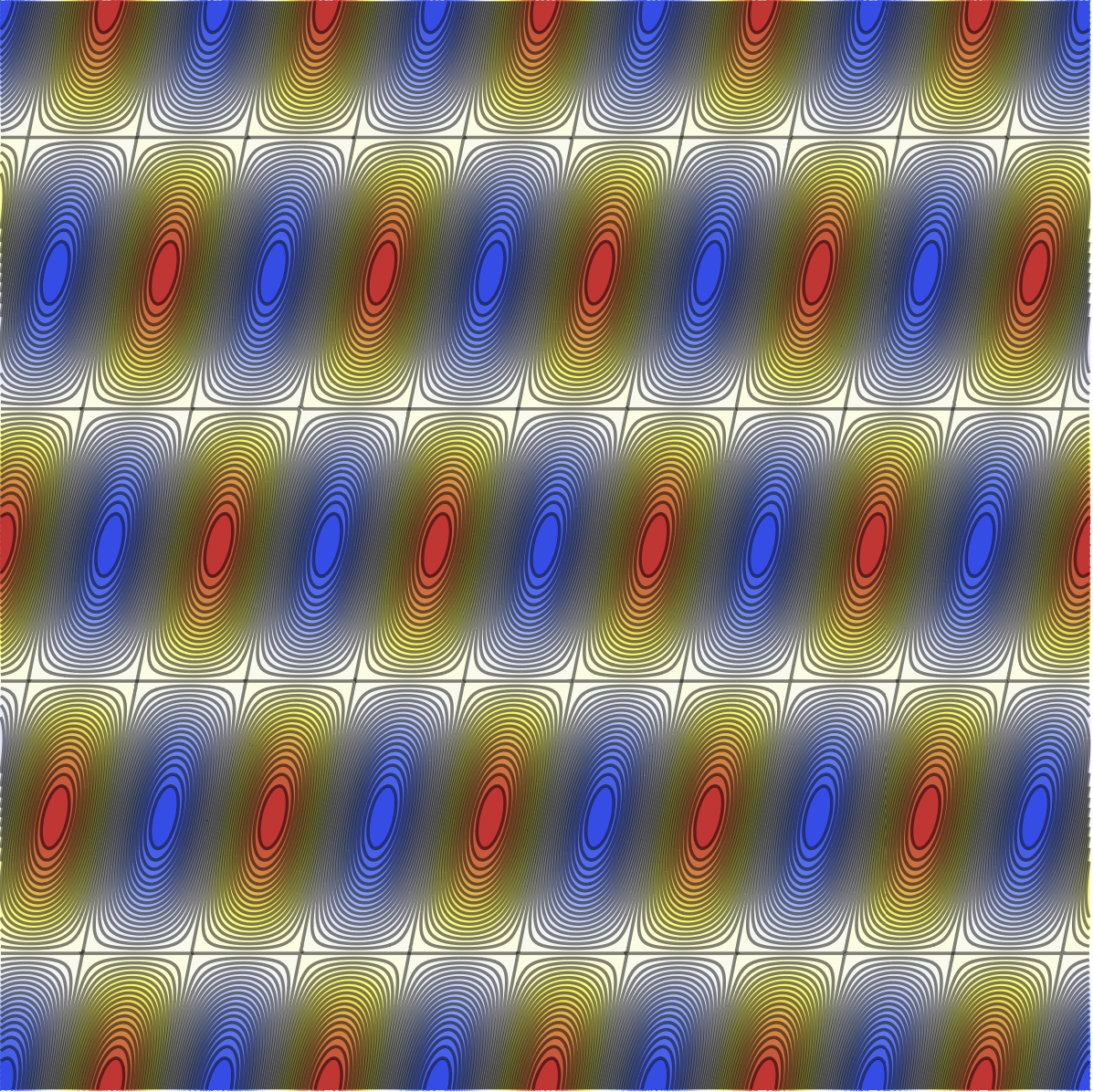}
    \end{subfigure}
    \caption{Contour plot of streamfunction $\psi$ of \cite{mis2} given by \eqref{steadystate}  with $n=6$ and $m=2$ is on the left. On the right, the corresponding test streamfunction $\phi$ given by \eqref{teststream} is shown.}   \label{fig:stream}
\end{figure} 
\end{proof} 

Theorem \ref{kflow1} provides an infinite family of stationary Euler states on the  flat two-torus with arbitrary aspect ratio which possess conjugate points. 
Recall also from the above discussion that Kolmogorov flows having very special unidirectional structure possess no conjugate points.
 We suspect that these flows are exceptional and that all other Kolmogorov flows possess conjugate points.
\begin{quest}\label{kquest}
Do all Kolmogorov flows on $\mathbb{T}_\alpha^2$ which are not unidirectional possess conjugate points?
\end{quest}
We add that it may be of interest to investigate Kolmogorov flows on more general domains 
(for example, the disk), which we also expect to generally possess conjugate points. In fact, there is interesting numerical work when the domain is the unit square which suggests the existence of cut points along the first eigenfunction of the Stokes operator  \cite{lan12}.

\section{Non-existence of conjugate points along Arnold stable flows} 

In this section we describe a class of steady states called Arnold stable flows. 
Recall that any stationary solution of the Euler equations is a critical point of the kinetic energy. 
The second variation of $E$ on tangent spaces 
to the coadjoint orbit $\mathcal{O}_{u_0}$ consisting of vector fields $w$ that are isovortical to $u_0$ 
was computed by Arnold \cite{arn} to obtain a criterion for Lyapunov stability of steady fluid flows 
in the mixed ``energy-enstrophy" norm" 
$\|v\|_{L^2} + \|\mathrm{curl}\, v\|_{L^2}$ 
under sufficiently small perturbations of the velocity field. 
Specifically, it is found that such flows can be constructed from solutions of the problem
\begin{align} \label{Euler1} 
\Delta \psi &= F(\psi) \qquad \text{in} \ M, 
\\ 
\psi &= {\rm (const.)}  \ \ \ \text{on} \ \partial M, \label{Euler2} 
\end{align} 
when $F$ enjoys certain properties described below 
and where $\Delta$ is the Laplace-Beltrami operator of the Riemannian metric $g$ on $M$.  
First note that the velocity is $u= \nabla^\perp \psi$ is automatically a solution of the Euler equation on $M$ 
with vorticity $\omega =F(\psi)$. 
The properties of $F$ which ensure nonlinear stability are either 
\begin{equation} 
-\lambda_1<F'(\psi) <0 \qquad \text{or} \quad 0< F'(\psi)<\infty 
\label{arnoldscond} 
\end{equation} 
where $\lambda_1>0$ is the smallest eigenvalue of $-\Delta$ in $M$, 
see \cite{arn} or \cite{arkh}; Thm. 4.3. 
Here we provide a simple argument that this stability forbids conjugate points in some settings. 
\begin{thm} \label{arnoldthm} 
Let $M\subset \mathbb{R}^2$ be the periodic channel, the annulus or the disk with the Euclidean metric. 
Then no Arnold stable flow on $M$ has a conjugate point. 
\end{thm} 
\begin{proof} 
The proof combines some known results. 
We first recall the following rigidity result: 
\begin{lemma}[Prop 1.1 of \cite{cdg}]\label{propA}
Let $(M,g)$ be a compact two-dimensional Riemannian manifold with smooth boundary $\partial M$.
Suppose that there exists a Killing field $K$ for $g$ which is tangent to $\partial M$. 
Let  $u= \nabla^\perp \psi \in C^2(M)$ be an Arnol'd stable state. 
Then $L_K \psi=0$ where $L_K$ is the Lie derivative along $K$. 
\end{lemma}
Lemma \ref{propA} immediately implies that on the channel with $K=e_{x_1}$ 
all Arnold stable states are shears $u= v(x_2) e_{x_1}$ and on the disk (or annulus) with $K=e_\theta$ 
all Arnold stable states are radial $u= v(r) e_\theta$. 

With these results in hand, we can now appeal to the following two results
\begin{lemma}[Corollary 6.1 or \cite{mis1}]\label{gcor}  
If $(M,g)$ has nonpositive sectional curvature then there are no conjugate points 
along any pressure constant geodesic in $\mathscr{D}_\mu^s(M)$.
\end{lemma} 
\begin{lemma}[Proposition 6.1 of \cite{pre1}]\label{scor}
If $(M,g)$ is flat then every $u = v(r) e_\theta$ defines a nonpositive $L^2$ curvature operator 
along the corresponding geodesic in $\mathscr{D}^s_\mu(M)$. 
\end{lemma} 
Since  plane-parallel (shear) flows on the channel  are pressure constant and the channel with the Euclidean metric has zero curvature, Lemma \ref{gcor} applies.  Lemma \ref{scor} directly applies to radial flows on the disk or annulus 
with the flat Euclidean metric. Since positive curvature is a necessary condition for conjugate points, these flows must therefore be free of conjugate points.  These two observations conclude the proof.
\end{proof} 

We would like to give some physical intuition behind the lack of conjugate points. 
Consider a steady state which is not isochronal, that is the travel time of particles around streamlines 
is not everywhere constant. Two nearby tracer particles in a `shearing' region will separate 
and a material curve connecting them will be ever elongated with length proportional to the time elapsed. 
On the disk or annulus, this curve would become indefinitely wound over time. 
This stretching represents a strong form of Lagrangian instability, and is an enemy for having conjugate points. 
Intuitively, Arnold stable flows cannot have conjugate points because their velocity (in fact, vorticity) field 
is structurally stable to nearby perturbations under the Euler dynamics so that shearing regions remain.  
As such, the stretching mechanism discussed above should persist. 
Theorem \ref{arnoldthm} establishes this for some simple geometries 
and strongly makes use of the structural rigidity of Arnold flows on those domains. 
We believe that this is a more general property of Arnold stable solutions\footnote{In 3D hydrodynamics 
both cut and conjugate points are plentiful. On the one hand (local) cut points exist 
on any sufficiently long geodesic \cite{shn}. 
On the other hand conjugate points cluster or even densely fill out finite geodesic segments \cite{emp}, \cite{prewash}.
This is in sharp contrast with 2D hydrodynamics and it is tempting to relate this phenomenon 
to the failure of Arnold's stability criterion due to indefiniteness of the quadratic form \cite{SV93}. 
In fact, it can be show that (with $v \in T_e\mathscr{D}_\mu^s(M)$  and  $ \|v\|_{L^2}^2=1$) we have 
\begin{align*} 
m_c^{u_0,v} 
&= 
-2E''_{u_0}(w) 
+ 
\big\langle [\mathrm{ad}_v, \mathrm{ad}_v^\ast] u_0, u_0 \big\rangle_{L^2} 
\end{align*} 
where $w = \mathrm{ad}_v^\ast u_0$,
which may prove useful in investigating this issue.},
 at least on flat manifolds, and ask more generally:
\begin{quest}\label{arnoldquest}
Let $(M,g)$ be a compact two-dimensional Riemannian manifold with smooth boundary $\partial M$.  
Can any Arnold stable solutions on $M$ have conjugate points?
\end{quest} 
%


\section{Cut points,  isochronal flows, and closed geodesics} 
\label{sec:cut} 


Closed geodesics are those whose initial and final points and tangent vectors coincide. 
In geometric hydrodynamics they correspond to flows which have periodic-in-time Lagrangian flowmap.  
A special class of these which correspond to stationary Euler solutions goes by the name of \emph{isochronal} flows. 
It is easy to see that an isochronal Euler flow on a simply connected domain must possess 
a stream function with a single elliptic stagnation point whose level sets foliate the domain by (topological) circles.  
All orbits of particle trajectories are thus periodic in time and the isochronality condition 
means that the period of rotation $T(c)$ for $c$ in the range of $\psi$ defined by 
\begin{align} 
T(c) &:=\oint_{\{ \psi=c\}} \frac{\rmd \ell}{|\nabla \psi|}, 
\end{align} 
is independent of $c$, the streamline traversed. 

\begin{example} \label{ex:loci1} 
\emph{An example of an isochronal flow is given by rotating the unit sphere $\mathbb{S}^2$ 
about the $z$-axis in $\mathbb{R}^3$. 
As described in Section \ref{sec:conj_points}, 
we can embed this flow in a one-parameter family of rotations $\varrho_s(t)$. 
Differentiating with respect to the parameter $s$ we obtain a Jacobi field 
along the geodesic in $\mathscr{D}_\mu^s(\mathbb{S}^2)$ corresponding to the original rotation 
which vanishes at $t=0$ and $t=2\pi$, see \cite{mis1}. 
Thus, we find an example of an isochronal flow where the first conjugate point 
coincides with the cut point. 
However, as in finite dimensions, we expect such situations to be rather exceptional.}
\end{example} 
\begin{example} \label{ex:loci2} 
\emph{
Another interesting explicit set of examples corresponds to elliptical vortices with constant vorticity, see \cite{MSY}.
Consider an elliptical domain with major axis $a>0$ and minor axis $b>0$. 
Let the stream function of the velocity be given by $\psi(x,y) = \frac{1}{2} (({x}/{a})^2+ ({y}/{b})^2)$. 
This defines a non-penetrating Euler solution with $u= \nabla^\perp \psi$ 
of constant vorticity 
$\omega = \Delta \psi=  {(a^2 + b^2)}/{a^2b^2}$.  
A short computation shows that the travel time $T(c) =2\pi ab$ is independent of $c$. 
Observe that any isochronal stationary flow on a simply connected planar domain possesses a cut point: 
two minimizing $L^2$ geodesics connecting two states are obtained simply by rotating left or right 
(switching overall direction of the velocity).}

\emph{As a special case of more general results of Preston \cite{pre1} (Proposition 6.1), 
the area-preserving diffeomorphism group at the radial member (occupying the disk or annulus) of the family of elliptical vortices above 
has non-positive curvature.\footnote{In fact, the curvature is identically zero.  This follows from the fact that the solid body rotation on the disk has velocity which is an isometry of the Euclidean space, so that the sectional curvatures must be non-negative \cite{mis1}. Combining these two facts gives the claim.} 
This precludes the existence of conjugate points thus yielding an example of an isochronal flow 
possessing a cut point while being free of conjugate points.}
\end{example} 

Interestingly, as soon as a simply connected planar domain is not radially symmetric, 
Theorem 3.3 (see also Proposition 3.1) of  \cite{pre1} implies that the sectional curvatures 
along any stationary state with isolated stagnation points cannot be all non-positive. 
In particular, given a stationary state $u=\nabla^\perp \psi$ on such a simply connected planar domain $M$, 
the sectional curvature is found to be positive in the direction $w = \psi  u$. 
An explicit computation gives 
\be\label{poscurv}
\big\langle \mathcal{R}_e(u,w) u, w \big\rangle_{L^2} 
= 
\big\| P_e (\psi \, u{\cdot} \nabla u) \big\|_{L^2}^2.
\ee
%
which 
vanishes if and only if $u$ satisfies $u{\cdot} \nabla |u|^2$, making $|u|$ constant on streamlines 
(in particular, along the boundary). It then follows that the domain is a disk or annulus. 
See also Theorem 1.9 of Hamel and Nadirashvili \cite{hn19}. 

In light of our focus, it is natural to ask whether conjugate points can be found in these directions. 
A simple calculation for $u =\nabla^\perp \psi$ and $w = \psi u$ shows that
\begin{align} 
m_c^{u,w} 
&= 
\| P_e (\psi \, u{\cdot}\nabla u)\|_{L^2}^2 
-  \| P_e (u{\cdot}\nabla ( \psi u))\|_{L^2}^2 
= 0 
\end{align} 
since $u{\cdot} \nabla \psi=0$.
Thus, we cannot apply Theorem \ref{conjpointthm} to conclude that $w = \psi u$ is a conjugate vector.  
We also cannot conclude from the above that there do not exist conjugate points in this direction. 
We therefore ask 
\begin{quest} 
Are vortices with constant strengths on elliptical domains free of conjugate points? 
\end{quest} 
If the answer is in the affirmative then we will have identified a family of steady states 
that have cut points that are not conjugate points and which \emph{do not} correspond to isometries of 
the underlying manifold $M$. 
One may wonder more generally if all isochronal flows on two-dimensional compact manifolds with boundary 
are free of conjugate points in $\mathscr{D}_\mu^s$. 

The example of a rigid rotation of $\mathbb{S}^2$ in Example \ref{ex:loci1} shows that the conjugate and the cut loci 
of the identity $e$ in $\mathscr{D}_\mu^s$ intersect\footnote{There are examples of compact 
Riemannian manifolds whose conjugate and cut loci are disjoint, cf. \cite{wei}.} 
while the example of the unit disk in Example \ref{ex:loci2} shows that the two geometric notions 
are essentially different and may therefore be describing different, though related, phenomena in fluid motion. 
This point is further amplified in the statement of the next theorem 
whose proof is a modification of an old argument of Klingenberg \cite{kling}.

\begin{thm} \label{thm:cut} 
Let $\eta \in \mathcal{U}$ be a fluid configuration  which lies on the cut locus of $e$ whose $L^2$ distance to $e$ is the smallest  among all such points for which there are (at least) two distinct vectors of minimal $L^2$ norm 
in the fibre $\mathrm{exp}_e^{-1}(\eta)$. 
If $\eta$ is not monoconjugate to $e$ along any $L^2$ geodesic then it is a midpoint of 
a closed geodesic based at $e$. 
\end{thm}

\begin{proof} 
Let $v_1$ and $v_2$ be two distinct vectors in $\mathrm{exp}_e^{-1}(\eta)$ of the smallest norm, say 
$\|v_1\|_{L^2} = \|v_2\|_{L^2} = 1$ 
and let 
$\gamma_1(t) = \mathrm{exp}_e{tv_1}$ 
and 
$\gamma_2(t) = \mathrm{exp}_e{tv_2}$ 
be the corresponding (minimal) geodesics joining $e$ to $\eta$ in $\mathscr{D}_\mu^s(M)$. 
This implies that $\eta$ must be the first cut point of $e$ along each geodesic 
(by any reasonable definition). 

We claim that $\dot{\gamma}_1(1) = -\dot{\gamma}_2(1)$. 
Observe that since $\eta$ is not monoconjugate to $e$ along either geodesic 
there exist disjoint open sets $\mathcal{V}_1$ and $\mathcal{V}_2$ in $T_e\mathscr{D}_\mu^s$ 
containing $\dot{\gamma}_1(0)$ and $\dot{\gamma}_2(0)$, respectively, 
such that 
$\mathrm{exp}_e\vert_{\mathcal{V}_1}$ and $\mathrm{exp}_e\vert_{\mathcal{V}_2}$
are diffeomorphisms onto some open neighbourhoods 
$\mathcal{W}_1$ and $\mathcal{W}_2$ of $\eta$. 
This follows from the inverse function theorem and the fact that $\mathrm{exp}_e$ is a Fredholm map of index zero, 
which implies that its derivative $d\,\mathrm{exp}_e(v_i)$ is a continuous linear bijection 
and hence an isomorphism of Banach spaces 
$T_e\mathscr{D}_\mu^s$ and $T_{\mathrm{exp}_e{v_i}}\mathscr{D}_\mu^s$ for each $i=1,2$. 

\begin{figure}[h!]
   \centering
\includegraphics[width=0.45\textwidth]{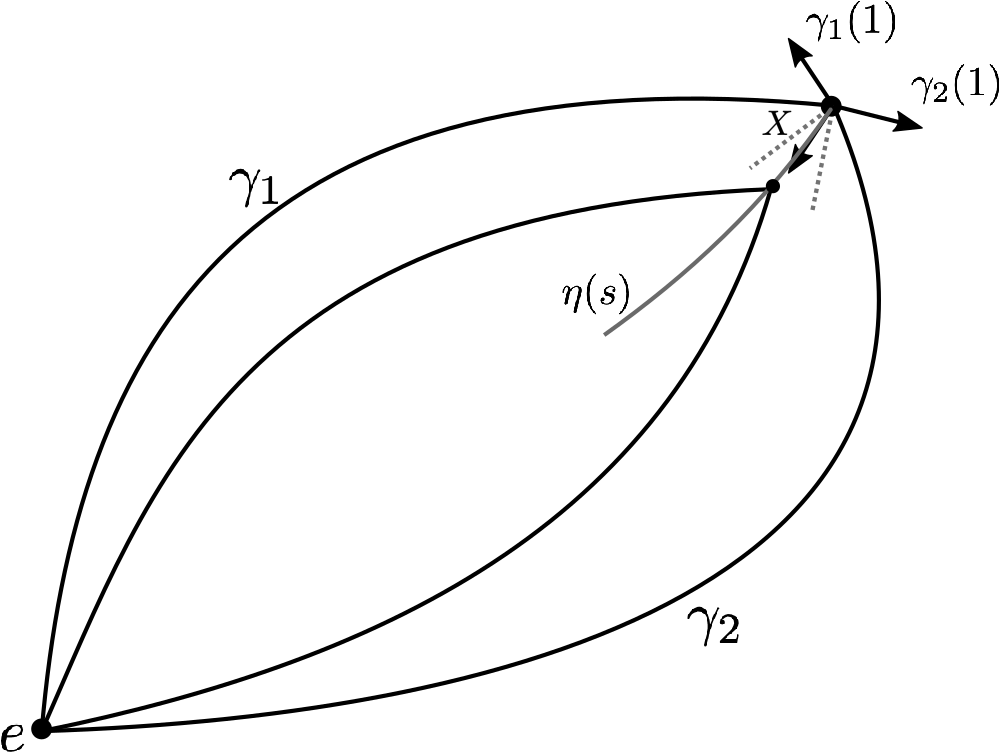}
    \caption{Illustration to the proof of Theorem \ref{thm:cut}. Grey dashed lines are perpendicular to 
                  the vectors $\gamma_1(1)$ and $\gamma_2(1)$.}   \label{fig:thm3}
\end{figure} 

Now, suppose that 
$\dot{\gamma}_1(1) \neq - \dot{\gamma}_2(1)$. 
Then we can pick a vector $X$ in $T_\eta\mathscr{D}_\mu^s$ such that 
$\langle X, \dot{\gamma}_1(1) \rangle_{L^2} < 0$ 
and 
$\langle X, \dot{\gamma}_2(1)\rangle_{L^2} < 0$ 
and let $\eta(s)$ be the $L^2$ geodesic with $\eta(0)=\eta$ and $\dot{\eta}(0)=X$. 
It follows that for a sufficiently small $s>0$ the point $\eta(s)=\mathrm{exp}_{\eta}{sX}$ 
must belong to $\mathcal{W}_1 \cap \mathcal{W}_2$ so that 
$$ 
w_{1,s} = \mathrm{exp}_e|_{\mathcal{V}_1}^{-1}(\eta(s)) 
\neq 
\mathrm{exp}_e|_{\mathcal{V}_2}^{-1}(\eta(s)) = w_{2,s} 
$$ 
in $T_e\mathscr{D}_\mu^s$. 
Therefore, we can define two one-parameter families of $L^2$ geodesics 
$\gamma_1(s,t) = \mathrm{exp}_e{tw_{1,s}}$ 
and 
$\gamma_2(s,t) = \mathrm{exp}_e{tw_{2,s}}$ 
from $e$ to $\eta(s)$ such that 
$\gamma_1(0,t)=\gamma_1(t)$ and $\gamma_2(0,t)=\gamma_2(t)$. 
From the first variation formula for the $L^2$ length functional 
(alternatively, one could work with the energy functional) 
we obtain 
\begin{align*} 
\frac{d}{ds}\Big|_{s=0}\mathcal{L}(\gamma_1(s)) 
&= 
\|\dot{\gamma}_1(0)\|_{L^2}^{-1} 
\bigg\{ \big\langle \frac{d\gamma_1}{ds}(0,t), \dot{\gamma}_1(t) \big\rangle_{L^2} \Big|_{t=0}^{t=1} 
- 
\int_0^1 \big\langle \frac{d\gamma_1}{ds}(0,t), \nabla_{\dot{\gamma}_1}\dot{\gamma}_1 (t) \big\rangle_{L^2} dt \bigg\} 
\\ 
&= 
 \langle X, \dot{\gamma}_1(1)\rangle_{L^2} < 0 
\end{align*} 
and similarly 
$ 
\frac{d}{ds}\big|_{_{s=0}}\mathcal{L}(\gamma_2(s)) < 0 
$ 
so that (for sufficiently small $s>0$) we have e.g., 
$$ 
\mathcal{L}(\gamma_1(s)) \leq \mathcal{L}(\gamma_2(s)) 
< 
\mathcal{L}(\gamma_1) = \mathcal{L}(\gamma_2) 
$$ 
which implies that there must be a cut point along $\gamma_2(s)$ 
that lies somewhere between $e$ and $\eta(s)$. 
However, since $\mathcal{L}(\gamma_2(s)) < \mathcal{L}(\gamma_1)$ 
this contradicts the fact that $\eta$ is a point which realizes the $L^2$ distance from $e$ to its cut locus. 

Finally, 
the theorem follows from the fact that there are no self-intersecting $L^2$ geodesics 
in $\mathscr{D}_\mu^s(M)$, see (iv) above and also \cite{smo}. 
\end{proof} 
%

\bibliographystyle{amsplain}

\end{document}